\documentclass{amsart}
\usepackage{amssymb,amsmath,latexsym}
\usepackage{amsthm}
\usepackage{fontenc}
\usepackage{amssymb}
\numberwithin{equation}{section}

\newtheorem{theorem}{Theorem}[section]

\begin{document}
\author{Alexander E. Patkowski}
\title{On properties of the generalized Davenport Expansion}

\maketitle
\begin{abstract} We study the continuity properties of a generalized Davenport Fourier expansion we recently discovered, by imposing conditions on the coefficients. We also put our expansion into perspective from the position of Appell sequences. \end{abstract}

% AMS keywords (used in AMS journals)

% AMS subject classifications (used in AMS journals)

\section{Introduction} Recall the fractional part function is given by $\{x\}=x-[x],$ where $[x]$ denotes the integer part of $x.$ In a recent paper [13], we proved a Fourier expansion which generalizes Davenport's celebrated result [4, eq.(2)], \begin{equation}\frac{1}{\pi}\sum_{n=1}^{\infty}\frac{A(n)}{n}\sin(2\pi n x)=-\sum_{n=1}^{\infty}\frac{a(n)}{n}\left(\{nx\}-\frac{1}{2}\right).\end{equation} There have been many follow-up papers on properties of these Fourier series [1, 3, 6, 7, 9, 16]. Here $a(n)$ is an arithmetic function such that $L(s)=\sum_{n\ge1}a(n)n^{-s}$ is holomorphic for $\Re(s)>1.$ Furthermore, $A(n)=F_0(n),$ where  $F_k(n):=\sum_{d|n}d^{-k}a(\frac{n}{d}).$ Now [13, Theorem 1.1] states that if $N\ge1,$ then for positive real numbers $x,$
\begin{equation}\sum_{n=1}^{\infty}\frac{a(n)}{n}\left(\{nx\}^{N}+N!\sum_{k=0}^{N-1}\frac{(-1)^k\zeta(-k)}{(N-k)!k!}\right)=
\sum_{n=1}^{\infty}a_n\sin(2 \pi n x)+\sum_{n=1}^{\infty}b_n\cos(2 \pi n x),\end{equation} where
 \begin{equation}a_n=-\frac{N!}{\pi n}\sum_{k=0}^{N-1}\frac{(-1)^kF_k(n)}{(N-k)!}\cos\left(\frac{\pi}{2}k\right),\end{equation}
 and \begin{equation}b_n=\frac{N!}{\pi n}\sum_{k=0}^{N-1}\frac{(-1)^kF_k(n)}{(N-k)!}\sin\left(\frac{\pi}{2}k\right).\end{equation} Here we have corrected a minor misprint to include the factor $(-1)^k$ in the finite sum contained in the terms of the series on the left side of (1.2). The proof we offered in [13] utilized a new Mellin transform and adapted Segal's proof [16]. One interesting example is when $a(n)=n^{-r},$ for $r>0.$ In this case $F_k(n)=n^{-r}\sum_{d|n}d^{r-k}=n^{-r}\sigma_{r-k}(n),$ and the left side of (1.2) becomes a refined form of the series
 $$\sum_{n=1}^{\infty}\frac{1}{n^{r+1}}\left(\{nx\}-\frac{1}{2}\right),$$ which was considered by Hardy and Littlewood in [5, pg.519]. \par It is a well-known property that a function need not be continuous everywhere to possess a Fourier series representation (e.g. [8, pg.66]). In particular, the Fourier series of such functions converge to the average of the limits from the right and left at these points of discontinuity. Several notable papers have made connections between the coefficients of Fourier series and the class of the function it represents [2, 11]. The main purpose of this paper is to identify which class the function on the left side of (1.2) belongs to when conditions are imposed on the coefficients (1.3) and (1.4). We also provide a connection between (1.2) to Appell sequences in the section that follows.
 \section{On the class of the generalized Davenport expansion}
 We say $v(\delta)$ is a modulus of continuity associated with a function $f:I\rightarrow\mathbb{R},$ if it has the property that when $x,y\in I$ for a closed interval $I,$ then $$|f(x)-f(h)|\le v(|x-h|).$$ Moreover, if $\lVert f\rVert$ is the maximum norm of a function $f,$ then
 \begin{equation}C^{v}:=\{f:\lVert f(x+h)-f(x)\rVert=O(v(h)) \}. \end{equation}
 It is known that series of the Davenport type are discontinuous at integral points, and in some cases rational points. For example, [7, pg.289, Theorem 6] states that (1.1) is discontinuous only at integral points when $a(n)=\Lambda(n),$ the von Mangoldt function [18, pg.4]. The proof of which involves a closed expression for the left hand side of (1.1) valid for $x\in(0,1).$ In order to apply theorems from N$\acute{e}$meth [11], we need to note first that the assumption of continuity is applied in the proofs therein to ensure that a bound exists for $|f(x+h)-f(x)|.$ Therefore, given the observation on the continuity of (1.1) at irrational points, we need to create a special class of a functions
 $$K^{v}:=\{x\in\mathbb{R}\setminus\mathbb{Q}\mid f: \lVert f(x+h)-f(x)\rVert=O(v(h)) \}.$$
 Note that if $v(h)=h^{\alpha},$ for a positive number $\alpha,$ then $f$ is Lipchitz of order $\alpha$ on $\mathbb{R}\setminus\mathbb{Q}.$ When $\alpha=0,$ we cannot infer anything about the continuity class of $f,$ but it is perhaps just as interesting that we can then say this function is locally bounded. Further, we have $C^{v}\subset K^{v}.$ We will utilize an altered form of a theorem due to N$\acute{e}$meth.

\begin{theorem} ([11, pg.86, Theorem 1]) Suppose that $c_n$ are non-negative sine or cosine coefficients of the function $f(x),$ then \begin{equation}\sum_{j=1}^{m}jc_j=O\left(mv\left(\frac{1}{m}\right)\right), \end{equation}
and
\begin{equation}\sum_{j=m}^{\infty}c_j=O\left(v\left(\frac{1}{m}\right)\right), \end{equation}
together ensure that $f(x)\in C^{v}.$
\end{theorem} 
Since in some cases our Fourier series coefficients will change sign infinitely often, we will need to state a related theorem which is known (see proof of [17, Theorem 3.1]).
\begin{theorem} Suppose that $c_n$ are sine or cosine coefficients of the function $f(x),$ then \begin{equation}\sum_{j=1}^{m}j|c_j|=O\left(mv\left(\frac{1}{m}\right)\right), \end{equation}
and
\begin{equation}\sum_{j=m}^{\infty}|c_j|=O\left(v\left(\frac{1}{m}\right)\right), \end{equation}
together ensure that $f(x)\in C^{v}.$
\end{theorem} 
\begin{proof} We mimic the proof of N$\acute{e}$meth's theorem [11, pg.89, eq.(35)] and Boas [2, pg.469], but assume that the coefficients may alternate in sign. In the case of cosine coefficients,
$$\begin{aligned}\left|f(x+2h)-f(x)\right|&=2\left|\sum_{j=1}^{\infty}c_j\sin\left(j(x+h)\right)\sin(jh) \right|\\
&\le 2\left|\sum_{j=1}^{[1/h]}c_j\sin(jh) \right|+2\left|\sum_{j=[1/h]}^{\infty}c_j \right|\\
&\le 2h\sum_{j=1}^{[1/h]}j|c_j|\left|\frac{\sin(jh)}{jh} \right|+2\sum_{j=[1/h]}^{\infty}|c_j|\\ 
&\le 2h\sum_{j=1}^{[1/h]}j|c_j|+2\sum_{j=[1/h]}^{\infty}|c_j|.\end{aligned} $$
Now applying the estimates (2.4) and (2.5) in Theorem 2.2 gives the result.
\end{proof}

In the case of N$\acute{e}$meth's theorem (Theorem 2.1), one is able to remove the absolute value in the second line of the proof after assuming $c_n\ge0,$ and noting $\sin\left(j(x+h)\right)\sin(jh)\le\sin(jh)$ is equivalent to $\sin\left(j(x+h)\right)\le 1,$ which is true for all real numbers $x$ and $h.$ The assumption throughout the paper [11] is that $f$ is a continuous function as noted earlier. Therefore, after noting (1.2) is continuous on $\mathbb{R}\setminus\mathbb{Q},$ it is a simple matter to amend the above theorem's proof to adjust for subsets of $\mathbb{R}.$ 
 We are now ready to state one of our main results once we define the function on the left hand side of (1.2) as 
 $$P_N(x):=\sum_{n=1}^{\infty}\frac{a(n)}{n}\left(\{nx\}^{N}+N!\sum_{k=0}^{N-1}\frac{(-1)^k\zeta(-k)}{(N-k)!k!}\right).$$
 
 \begin{theorem}
 Let $a(n)$ be a non-negative, monotonically increasing function for which $L(s)$ is analytic for $\Re(s)>1.$ Then for any such $a(n),$ $P_N(x)\notin K^{v}.$ 
 \end{theorem}
 \begin{proof}First, note that if $a(n)$ is monotonically increasing, then $F_k(n)=\sum_{d|n}d^{-k}a(\frac{n}{d})=n^{-k}\sum_{d|n}d^{k}a(d)=O(na(n))$ since $n$ is always a divisor of $n,$ and there are at most $n$ divisors of $n.$ Now for (2.4) we compute for the sine coefficient case,
 $$\begin{aligned}\sum_{j=1}^{m}j|a_j|&=\sum_{j=1}^{m}\left|\frac{N!}{\pi }\sum_{k=0}^{N-1}\frac{(-1)^k}{(N-k)!}\cos\left(\frac{\pi}{2}k\right)F_k(j)\right|\\
 &=O\left(\frac{N!}{\pi }\sum_{k=0}^{N-1}\frac{1}{(N-k)!}\sum_{j=1}^{m}ja(j)\right)\\
 &=O\left(m^2a(m)\right).\end{aligned} $$
 This implies we may potentially have $v(n)=h^{-1}a(1/h),$ which ultimately says $P_N(x)$ is divergent. The case (2.5) fails however, since $\sum_{j=m}^{\infty}j^{-1}F_k(j)$ is unbounded. A similar argument applies to the case for cosine coefficients. Hence, by Theorem 2.2 the result follows. \end{proof}
 A nice example of the above theorem is the choice $a(n)=\log(n).$ In this case, we also have $P_N(x)\notin K^{v},$ since $v(h)=h^{-1}\log(1/h)\rightarrow\infty$ as $h\rightarrow0^{+}.$ We now consider appealing to a result of Segal [15].
\begin{theorem} ([15, Theorem 1]) Let $a(n)$ be an arithmetic function such that $$\sum_{n=1}^{\infty}\left|\frac{1}{n}F_0(n)\right|=S<\infty.$$ If $\lim_{n\rightarrow\infty}b(n)=L<\infty,$ for an arithmetic function $b(n),$ then
$$\sum_{n=1}^{\infty}\frac{1}{n}\sum_{d|n}a(d)b\left(\frac{n}{d}\right)=SL.$$
\end{theorem}
In the below theorems, we will take $b(n)=n^{-k},$ $k\ge0$ in Theorem 2.4 noting that Dirichlet convolution is commutative, to obtain results on (1.2). 

\begin{theorem}Let $a(n)$ be an arithmetic function such that $F_0(n)=O(1),$ and $$\sum_{n=1}^{\infty}\left|\frac{1}{n}F_0(n)\right|=S<\infty.$$ Then $P_N(x)\in K^{v},$ where $v(h)=1,$ on $\mathbb{R}\setminus\mathbb{Q}.$
\end{theorem} 

\begin{proof} In order for the series in the theorem to converge absolutely to $S,$ it is possible that $n^{-1}F_0(n)$ is $O(n^{-\epsilon}),$ for some $\epsilon>1,$ or zero almost everywhere. In most of these cases $F_0(n)=O(1).$ (That is, the implied constant is taken as the maximum number from the finite set $\{n: F_0(n)\neq 0\}$ in the latter case.) For (2.4) we compute for the sine coefficient case,
 $$\begin{aligned}\sum_{j=1}^{m}j|a_j|&=\sum_{j=1}^{m}\left|\frac{N!}{\pi }\sum_{k=0}^{N-1}\frac{(-1)^k}{(N-k)!}\cos\left(\frac{\pi}{2}k\right)F_k(j)\right|\\
  &=O\left(\sum_{j=1}^{m}|F_0(j)|\right)\\
 &=O\left(\sum_{j=1}^{m}1\right)\\
 &=O(m).\end{aligned} $$
 By our assumption of absolute convergence, we have that 
 $$\begin{aligned}\sum_{j=m}^{\infty}\left|\frac{F_k(j)}{j}\right|&\le S-\sum_{j=1}^{m-1}\left|\frac{F_k(j)}{j}\right| \\
 &\le S+\sum_{j=1}^{m-1}\frac{\sum_{d|j}A_k(d)F_0(\frac{j}{d})}{j}\\
 &=O(1).\end{aligned}$$
Here we set $A_k(j)=\sum_{d|j}d^{-k}\mu(\frac{j}{d}).$ The M$\ddot{o}$bius function $\mu(n)$ is $(-1)^m$ if $n$ can be represented as the product of $m$ different primes, and is equal to zero if any factor is represented as a power greater than $1$ [18, pg.3]. The result now follows from Theorem 2.2.
\end{proof}
Next we give an example involving the M$\ddot{o}$bius function for coefficients.
\begin{theorem}
If $a(n)=\mu(n),$ the M$\ddot{o}$bius function, then $P_N(x)\in K^{v},$ for $N=1$ where $v(h)=1,$ on $\mathbb{R}\setminus\mathbb{Q}.$
\end{theorem}
\begin{proof}
We will require Ramanujan sums and some of their properties. From [18, pg.10, eq.(1.5.5)] we have
\begin{equation}\frac{1}{\zeta(s)}\sum_{n=1}^{\infty}\frac{c_k(n)}{n^s}=\sum_{d|k}\mu(\frac{k}{d})d^{1-s},\end{equation} for $\Re(s)>1.$ From [14, pg.185, eq.(7.3)]
\begin{equation}\frac{\sigma_{r-1}(n)}{n^{r-1}\zeta(r)}=\sum_{m=1}^{\infty}\frac{c_m(n)}{m^r} \end{equation} where $\sigma_r(n)=\sum_{d|n}d^{r},$ which when $r=1$ gives [14, pg.185, eq.(7.4)] \begin{equation}\sum_{m=1}^{\infty}\frac{c_m(n)}{m}=0. \end{equation} 
Now for (2.4) we compute for the sine coefficient case using (2.6),
 $$\begin{aligned}\sum_{j=1}^{m}j|a_j|&=\sum_{j=1}^{m}\left|\frac{N!}{\pi }\sum_{k=0}^{N-1}\frac{(-1)^k}{(N-k)!}\cos\left(\frac{\pi}{2}k\right)F_k(j)\right|\\
 &=O\left(\frac{N!}{\pi }\sum_{k=1}^{N-1}\frac{1}{(N-k)!}\frac{1}{\zeta(k+1)}\sum_{j=1}^{m}\left|\sum_{n=1}^{\infty}\frac{c_j(n)}{n^{k+1}}\right|\right)\\
 &=O\left(\sum_{j=1}^{m}1\right)\\
 &=O(m)\\
 &=O(m\log(m)).\end{aligned} $$ In the third line we used $|c_k(n)|\le \sigma_1(n),$ [18, pg.10]. We may use this inequality for (2.5) to similarly compute 
 $$\begin{aligned}\sum_{j=m}^{\infty}|a_j|&=\sum_{j=m}^{\infty}\left|\frac{N!}{\pi }\sum_{k=0}^{N-1}\frac{(-1)^k}{(N-k)!}\cos\left(\frac{\pi}{2}k\right)\frac{F_k(j)}{j}\right|\\
 &=O\left(\sum_{j=1}^{\infty}\left|\frac{F_0(j)}{j}\right| \right)+O\left(\frac{N!}{\pi }\sum_{k=1}^{N-1}\frac{1}{(N-k)!}\frac{1}{\zeta(k+1)}\sum_{j=1}^{m-1}\frac{1}{j}\left|\sum_{n=1}^{\infty}\frac{c_j(n)}{n^{k+1}}\right|\right)\\
 &=O(1)+O\left(\sum_{j=1}^{m-1}\frac{1}{j}\right)\\
 &=O(\log(m-1))\\
 &=O(\log(m)).\end{aligned} $$ 
Here we have employed [12, pg.135, eq.(4.3.4)]
$$\sum_{j=1}^{n}\frac{1}{j}=\log(n)+\gamma+\frac{1}{2n}+o(1),$$ and the fact that $F_0(n)=0,$ unless $n=1,$ for the M$\ddot{o}$bius function. Hence we have that $v(h)=\log(1/h).$ In this instance we can improve our bound as $h\rightarrow0^{+}$ by observing that
$$\begin{aligned} \sum_{j=m}^{\infty}|a_j|&=\sum_{j=m}^{\infty}\left|\frac{N!}{\pi }\sum_{k=0}^{N-1}\frac{(-1)^k}{(N-k)!}\cos\left(\frac{\pi}{2}k\right)\frac{F_k(j)}{j}\right|\\ 
&=\sum_{j=m}^{\infty}\left|\frac{N!}{\pi }\sum_{k=0}^{N-1}\frac{(-1)^k}{(N-k)!}\cos\left(\frac{\pi}{2}k\right)\frac{\prod{p|j}(1-p^{k})}{j^{k+1}}\right|. \end{aligned}$$

upon invoking $$\frac{F_k(j)}{j}=\frac{\prod{p|j}(1-p^{k})}{j^{k+1}}.$$ It follows that since the inner sum is $\sim C/j$ for large $j$ it is bounded only when the sum is empty at $N=1$ by $F_0(n)=0.$ When replacing cosine in the same sum with $\sin\left(\frac{\pi}{2}k\right),$ $N=1$ gives an empty sum as well. Since the first finite sum up to $m$ was estimated to be $O(m),$ we may maximize the tail end second sum to be $O(1)$ to arrive at $v(h)=1.$
\end{proof}
The Louiville function $\lambda(n)$ is $(-1)^m,$ if $n$ has $m$ prime factors, and if $p^j$ for a prime $p$ is such a factor then it is counted $j$ times [18, pg.6].
\begin{theorem}
If $a(n)=\lambda(n),$ the Louiville function, then $P_N(x)\in K^{v},$ for $N=1$ where $v(h)=1,$ on $\mathbb{R}\setminus\mathbb{Q}.$
\end{theorem}
\begin{proof} The proof is very similar to the previous theorem. We again observe the first sum has
$\sum_{j=1}^{m}j|a_j|=O(m),$
for large $m.$ Noting that with $a(n)=\lambda(n)$ we have
$$\sum_{n=1}^{\infty}\frac{F_k(n)}{n^{s}}=\frac{\zeta(s+k)\zeta(2s)}{\zeta(s)},$$
for $\Re(s)>1.$ If we set $s=1,$ then the sum is $0$ giving us
$$\sum_{n=m}^{\infty}\frac{F_k(n)}{n}=-\sum_{n=1}^{m-1}\frac{F_k(n)}{n}.$$ By the Mellin-Perron formula this is easily estimated as $O(m^{-1/2})$ when $k=0$ due to the pole at $s=-\frac{1}{2}.$ Note also that in the Louiville case $F_0(n)$ is $1$ if $n$ is a square and zero otherwise. Hence the second tail sum over $m$ applied in Theorem 2.2 is certainly $O(1)$ provided $N=1,$ giving rise to $v(h)=1.$
\end{proof}
We mention that it may be possible to improve our estimates, but it would require a more flexible theorem than Theorem 2.1 or Theorem 2.2. This is because our  choices of $\mu(n)$ and $\lambda(n)$ would allow for conditional convergence, and those theorems do not account for this property.
 \section{Appell sequences}
An Appell sequence $w_n(x),$ is a sequence which satisfies [10]
\begin{equation}\mathfrak{T}(t)e^{zt}=\sum_{n=0}^{\infty}\frac{w_n(z)}{n!}t^n, \end{equation}
where $\mathfrak{T}(t)$ is analytic around $t=0,$ and therefore possesses a Taylor series. Put
$$G(t,x)=\sum_{k=0}^{\infty}(-t)^k\left(\frac{\cos(\frac{\pi}{2}k)}{\pi}\sum_{n=1}^{\infty}\frac{F_k(n)}{n}\sin(2 \pi n x)-\frac{\sin(\frac{\pi}{2}k)}{\pi}\sum_{n=1}^{\infty}\frac{F_k(n)}{n}\cos(2 \pi n x)\right),$$ for $|t|<1.$ It can be seen that there exists $a(n)$ for which $G(t)$ is analytic in a neighborhood of $0.$ One example is $a(n)=\lambda(n),$ which gives $G(t)=O((1-t)^{-1})$ since its coefficients are then $O(1).$ Equation (1.2) may be realized as the coefficient of $t^{N}$ of the formal expansion for $|t|<1,$
\begin{equation}G(t,x)\left(e^{t}-1\right)=\sum_{n=0}^{\infty}\frac{\bar{P}_n(x)}{n!}t^n,\end{equation}
where $\bar{P}_n(x)=P_n(x)$ as in the previous section for $n\ge1,$ and $P_0(x)=0.$ Based on [10], we may introduce a new variable $z$ and observe (3.2) is equivalent to 
\begin{equation}G(t,x)\frac{\left(e^{t}-1\right)}{t-2\pi i k}=\sum_{n=0}^{\infty}\left(\int_{0}^{1}e^{-2\pi i kz}\bar{P}_{n}(x,z)dz\right)\frac{t^n}{n!},\end{equation}
for some $\bar{P}_n(x,z)$ where
\begin{equation}G(t,x)e^{zt}=\sum_{n=0}^{\infty}\bar{P}_n(x,z)\frac{t^n}{n!}.\end{equation}
If we multiply both sides of (3.2) by $(t-2\pi i k)^{-1},$ we find by [10, pg.846]
\begin{equation}G(t,x)\frac{\left(e^{t}-1\right)}{t-2\pi i k}=-\sum_{n=0}^{\infty}s_{k,n}t^n,\end{equation}
where
$$s_{k,n}=\sum_{j=0}^{n}(2\pi ik)^{j-n-1}\frac{\bar{P}_j(x)}{j!}.$$
Hence equating coefficients of $t^n$ in (3.3) and (3.5),
$$\int_{0}^{1}e^{-2\pi i kz}\bar{P}_{n}(x,z)dz=-n!s_{k,n},$$ and we have proven the following theorem.
\begin{theorem} Let $x$ be restricted to the region for which $\bar{P}_j(x)$ converges, relative to the choice of $a(n).$ The Appell sequence defined by (3.4), satisfy
$$\bar{P}_{n}(x,z)=\sum_{k\in\mathbb{Z}\setminus\{0\}}d_{k,n}e^{2\pi ikz},$$
pointwise for $z\in(0,1)$ where $d_{k,n}=-n!s_{k,n},$ and 
$$s_{k,n}=\sum_{j=0}^{n}(2\pi ik)^{j-n-1}\frac{\bar{P}_j(x)}{j!}.$$

\end{theorem}

1390 Bumps River Rd. \\*
Centerville, MA
02632 \\*
USA \\*

ul. A. E. Ody\'{n}ca 47 \\*
02-606 Warsaw\\*
Poland\\*
E-mail: alexpatk@hotmail.com, alexepatkowski@gmail.com
 \\*
Competing interests: The author declares none.
 \\*
Funding statement: The author did not receive support for the submitted work.
\end{document}